\documentclass[a4paper,11pt]{article}
\usepackage{titlesec}
\usepackage{amsthm}
\usepackage{amssymb}
\usepackage{color}
\usepackage{mathtools}
\usepackage{geometry}
\usepackage{algorithm2e}
\usepackage[T1]{fontenc}

\usepackage{verbatim}
\numberwithin{equation}{section}

\newtheorem{theorem}{Theorem}[section]
\newtheorem{corollary}[theorem]{Corollary}

\theoremstyle{definition}

\newtheorem{remark}[theorem]{Remark}

\newcommand{\bfx}{\boldsymbol{x}}
\newcommand{\bfy}{\boldsymbol{y}}
\newcommand{\bfz}{\boldsymbol{0}}

\usepackage{array}
\newcolumntype{C}[1]{>{\centering\arraybackslash}m{#1}}
\usepackage{graphics}
\usepackage{amssymb, latexsym, amsmath, amsfonts, epsfig}
\usepackage{bm}
\usepackage{graphicx}
\usepackage{color}
\usepackage{epstopdf}
\usepackage{comment}
\usepackage{soul}
\usepackage[normalem]{ulem}
\usepackage{float}
\usepackage{MnSymbol}

\usepackage{caption}
\date{\vspace{-6ex}}

\begin{document}

\newcommand{\Question}[1]{{\marginpar{\color{blue}\footnotesize #1}}}
\newcommand{\blue}[1]{{\color{blue}#1}}
\newcommand{\red}[1]{{\color{red} #1}}

\newif \ifNUM \NUMtrue

\title{Exact Eigenvalues and Eigenvectors for Some n-Dimensional Matrices}
\author{Quanling Deng\thanks{School of Computing, Australian National University, Canberra, ACT 2601, Australia. E-mail addresses: quanling.deng@anu.edu.au; qdeng12@gmail.com.}
%\and
%Alexandre Ern\thanks{University Paris-Est, CERMICS (ENPC), 77455 Marne la Vall\'ee cedex 2, and INRIA Paris, 75589 Paris, France. E-mail address: alexandre.ern@enpc.fr}
}

\maketitle

\begin{abstract}
Building on previous work that provided analytical solutions to generalised matrix eigenvalue problems arising from numerical discretisations, this paper develops exact eigenvalues and eigenvectors for a broader class of $n$-dimensional matrices, focusing on non-symmetric and non-persymmetric matrices. These matrices arise in one-dimensional Laplacian eigenvalue problems with mixed boundary conditions and in a few quantum mechanics applications where standard Toeplitz-plus-Hankel matrix forms do not suffice. By extending analytical methodologies to these broader matrix categories, the study not only widens the scope of applicable matrices but also enhances computational methodologies, leading to potentially more accurate and efficient solutions in physics and engineering simulations. 

%Furthermore, we explore an inverse-type problem by assuming the forms of eigenvectors and constructing the matrices that exhibit these eigenvectors.

%\textbf{Mathematics Subjects Classification}: 65N15, 65N30, 65N35, 35J05
\end{abstract}
%

%\begin{keywords} 

\section{Introduction} \label{sec:intr}

The matrix eigenvalue problem (MEVP) is to find the eigenpairs $(\lambda \in \mathbb{C}, \bfx \in \mathbb{C}^n)$ such that
\begin{equation} \label{eq:evp}
A \bfx = \lambda \bfx
\end{equation} 
while the generalised MEVP (GMEVP) is to find the eigenpairs $(\lambda \in \mathbb{C}, \bfx \in \mathbb{C}^n)$ such that
\begin{equation} \label{eq:gevp}
A \bfx = \lambda B \bfx,
\end{equation}
where $A, B \in \mathbb{C}^{n\times n}$. 
While various numerical methods are available in the literature, this paper focuses on finding the analytical (interchangeably, exact) eigenvalues and eigenvectors for certain structured MEVP and GMEVP with arbitrarily large $n$.

We provide a very brief background for the problem. MEVPs and GMEVPs are pivotal in various scientific and engineering disciplines \cite{strang1988linear, meyer2000matrix}. Recent developments have achieved analytical solutions for eigenpairs of certain structured matrices--such as Toeplitz-plus-Hankel and block-diagonal matrices—that typically arise from the discretisation of differential equations \cite{hughes2008duality, hughes2012finite, hughes2014finite,calo2019dispersion, deng2018dmm, deng2020boundary}. While these solutions offer robust frameworks for symmetric or persymmetric matrices, they fall short of addressing the complexities presented by non-persymmetric and non-symmetric matrices frequently encountered in more general settings.

This paper builds upon the classic and more recent results obtained for special matrices and their generalisations to structured forms (typically symmetric and persymmetric) \cite{solary2013finding, barrera2017asymptotics, da2019eigenpairs, da2007eigenvalues,willms2008analytic,trench1999eigenvalues,kouachi2006eigenvalues,yueh2005eigenvalues,fasino1988spectral, losonczi1992eigenvalues, chang2009exact}. We refer to \cite{deng2021analytical} for a relatively more detailed literature study of matrices where exact eigenpairs are available. 
The work \cite{deng2021analytical} provides exact eigenpairs for five sets of symmetric and persymmetric matrices as well as their generalisations. 
In this paper, we extend the results and provide the exact eigenvalue and eigenvectors for several new sets of $n$-dimensional non-symmetric and non-persymmetric matrices. To the author's best knowledge, exact eigenpairs were not available in the literature for these new sets of matrices. 
 The generalised approach retains the rigour of analytical methods while broadening their applicability to matrices encountered in advanced applications, including those that arise from the discretisations of PDEs with non-standard boundary conditions and other fields of science (e.g, the matrices in \cite{sony2024strictly} arising from Quantum Physics).
The results presented herein not only demonstrate the analytical derivability of eigenpairs for more complex matrix structures but also propose novel methodologies for their computation.

The rest of the article is organized as follows.
Section~\ref{sec:snp} presents the main results for matrices that are both symmetric and non-persymmetric, 
while Section~\ref{sec:nsnp} presents the main results for matrices that are both non-symmetric and non-persymmetric.
%In Section \ref{sec:mg}, we consider an inverse-type problem: assume the forms of eigenvectors and construct the matrices.
Concluding remarks are given in Section~\ref{sec:conclusion}.

\section{Symmetric and non-persymmetric matrices} \label{sec:snp}

In this section, we establish exact eigenvalues and eigenvectors for a few sets of $n$-dimensional matrices that are both symmetric and non-persymmetric.
For simplicity, we adopt the notation in \cite{deng2021analytical}.
We denote matrices and vectors by uppercase and lowercase bold letters, respectively. 
In particular, let $A \in \mathbb{C}^{n\times n}$ be a $n\times n$ square matrix with entries denoted as $A_{jk},j,k=1,\cdots,n,$ and $\bfx = (x_1, x_2, \cdots, x_n)^T\in \mathbb{C}^{n}$ be a vector with entries denoted as $x_j,j=1,\cdots,n,$ in the complex field. 
Furthermore, 
we denote by $A^{(\xi,m)}$ a structured matrix that the entries depend on a sequence of parameters $\xi = (\xi_0, \xi_1, \cdots, \xi_m).$ 
Herein, $\xi$ is a generic parameter that specifies the matrix; for example, $A^{(\beta,m)}$ is a structured matrix depending on the parameters $\beta = (\beta_0, \beta_1, \cdots, \beta_m).$ 
The superscript $\cdot^{(\xi,m)}$ is omitted when the context is clear. 
We denote by $T^{(\xi,m)} = (T^{(\xi,m)}_{j,k}) \in \mathbb{C}^{n\times n}$ a symmetric, persymmetric, and diagonal-structured Toeplitz matrix with entries
\begin{equation} \label{eq:T}
T^{(\xi,m)}_{j,j+k} = 
\begin{cases}
\xi_{|k|}, \quad & |k| \le m, \quad k \in \mathbb{Z}, \quad  j = 1,\cdots, n, \\
0, \quad & \text{otherwise},
\end{cases}
\end{equation}
where $m\le n-1$ specifies the matrix bandwidth. 
Explicitly, the matrix $T^{(\alpha,m)}$ can be written as
\begin{equation*}
\begin{aligned}
T^{(\alpha,m)} & = 
\begin{bmatrix}
\alpha_0   & \alpha_1   & \cdots   & \alpha_m  \\
\alpha_1    & \alpha_0 & \alpha_1  & \cdots  & \alpha_m   \\
\vdots  & \alpha_1  & \alpha_0 & \alpha_1  & \cdots   & \alpha_m   \\
\alpha_m  & \cdots  & \alpha_1  & \alpha_0 & \alpha_1  & \cdots   & \alpha_m   \\
 &\ddots & \ddots & \ddots & \ddots & \ddots & \ddots & \ddots  \\
&  & \alpha_m  & \cdots  &\alpha_1  & \alpha_0 & \alpha_1 & \cdots \\
&& &  \alpha_m  & \cdots  & \alpha_1  & \alpha_0 & \alpha_1   \\
& && &  \alpha_m  & \cdots  & \alpha_1  & \alpha_0    \\
\end{bmatrix}_{n\times n},
\end{aligned}
\end{equation*}
where the empty spots are zeros.

\subsection{Toeplitz-plus-Hankel matrices}

In this section, we develop exact eigenpairs for some Toeplitz-plus-Hankel matrices that are not persymmetric (and symmetric). 
The symmetric and persymmetric cases have been studied in \cite{deng2021analytical}. The main idea is to seek eigenvectors that can be written in trigonometric forms with subtle treatments near the boundary entries. 
Herein, our main contribution is that we extend the results to non-persymmetric matrices.

To derive the exact eigenpairs for the first set of matrices, we denote by $H^{(\alpha,m)} = (H^{(\alpha,m)}_{j,k}) \in \mathbb{C}^{n\times n}$ a Hankel matrix with entries
\begin{equation} \label{eq:h1}
\begin{cases}
H^{(\alpha,m)}_{j,k} & =  \alpha_{j+k-1}, \quad   k = 1, \cdots, m-j+1, \quad j =1, \cdots, m, \quad 1\le m\le n, \\
H^{(\alpha, m)}_{n-j+1,n-k+1} & =  \alpha_{j+k}, \quad  k = 1, \cdots, m-j, \quad j =1, \cdots, m-1, \quad 2\le m\le n, \\
H^{(\alpha,m)}_{j,k} & = 0, \quad  \text{otherwise},
\end{cases}
\end{equation}
which can be explicitly written as
\begin{equation*}
\begin{aligned}
H^{(\alpha,m)} & = 
\begin{bmatrix}
\alpha_1   & \alpha_2 & \cdots   & \alpha_m & 0 & \cdots \\
\alpha_2    &  \cdots  & \alpha_m & 0 & \ddots  \\
%\alpha_4  &  \cdots  & \alpha_m & 0 & \ddots  \\
\vdots  & \alpha_m  &0 & \ddots \\
\alpha_m &0  &\ddots  \\
0 &\ddots  \\
\vdots  \\
& & & & &  &  & \alpha_m   \\
& & & & &  & \alpha_m  & \vdots \\
& & & & &  \alpha_m & \ddots & \alpha_3 \\
& & & & \alpha_m & \cdots & \alpha_3 & \alpha_2
\end{bmatrix}_{n\times n}.
\end{aligned}
\end{equation*}
Apparently, this matrix is non-persymmetric. 
When $m=1$, the entries at the bottom left corner of the matrix are zeros. 
For example, when $m=2$ and $n\ge4$, then the matrix $A=T^{(\xi,m)} - H ^{(\xi,m)}$ is of the explicit form
\begin{equation*}
\begin{aligned}
A & = 
\begin{bmatrix}
\alpha_0 - \alpha_1  & \alpha_1 - \alpha_2  & \alpha_2     \\
\alpha_1 - \alpha_2   & \alpha_0 & \alpha_1  & \alpha_2    \\
\alpha_2  & \alpha_1  & \alpha_0 & \alpha_1  & \alpha_2      \\
&\ddots & \ddots & \ddots & \ddots & \ddots & \ddots \\
&  &  \alpha_2  &\alpha_1  & \alpha_0 & \alpha_1 & \alpha_2 \\
&& &   \alpha_2  & \alpha_1  & \alpha_0 & \alpha_1   \\
& && &   \alpha_2  & \alpha_1  & \alpha_0 - \alpha_2   \\
\end{bmatrix}_{n\times n}.
\end{aligned}
\end{equation*}
With the matrix settings in mind, we have the following results.

\begin{theorem}[Exact eigenvalues and eigenvectors, set 1]\label{thm:set1}
Let $n\ge2, 1\le m \le n-1$ and
$
A = T^{(\alpha,m)} - H^{(\alpha,m)}, B = T^{(\beta,m)} - H^{(\beta,m)}
$ 
with $T^{(\xi,m)}$ and $H^{(\xi,m)}$, $\xi = \alpha, \beta$,  defined in \eqref{eq:T} and \eqref{eq:h1}, respectively.
Assume that $B$ is invertible. Then, the GMEVP \eqref{eq:gevp} has eigenpairs $(\lambda_j, \bfx_j)$ where
\begin{equation} \label{eq:set1}
\lambda_j = \frac{ \alpha_0 + 2 \sum_{l=1}^{m} \alpha_l \cos(l j\pi h) }{ \beta_0 + 2 \sum_{l=1}^{m} \beta_l \cos(l j\pi h)  }, \quad x_{j,k} = C \sin\left( j \pi (k-1/2) h \right), \quad h = \frac{1}{n+1/2}, \  j, k =1,2,\cdots, n,
\end{equation}
where $C \ne 0$ is a constant. 
\end{theorem}

\begin{proof}
Following \cite{deng2021analytical}, we seek eigenvectors of the form $C\sin\left( j \pi (k-1/2) h \right)$. Using the trigonometric identity $\sin(\phi \pm \psi) = \sin(\phi) \cos(\psi) \pm \cos(\phi) \sin(\psi)$, one can verify that each row of the GMEVP \eqref{eq:gevp}, $\sum_{k=1}^n A_{ik} x_{j,k} = \lambda \sum_{k=1}^n B_{ik} x_{j,k}, i=1,\cdots, n$,  reduces to $\alpha_0 + 2 \sum_{l=1}^{m} \alpha_l \cos(l j\pi h) = \lambda \big( \beta_0 + 2 \sum_{l=1}^{m} \beta_l \cos(l j\pi h) \big)$, which is independent of the row number $i$. Thus, the eigenpairs $(\lambda_j, x_j)$ given in \eqref{eq:set1} satisfies \eqref{eq:gevp}. 
The GMEVP has at most $n$ eigenpairs and the $n$ eigenvectors are linearly independent. This completes the proof.
\end{proof}

\begin{corollary}[Exact eigenvalues and eigenvectors, set 2]\label{thm:set2}
Let $n\ge2, 1\le m \le n-1$ and $T^{(\xi,m)}$ and $H^{(\xi,m)}$, $\xi = \alpha, \beta$,  be defined in \eqref{eq:T} and \eqref{eq:h1}, respectively.
Let $G^{(\xi,m)}$ be the anti-diagonal transpose of $H^{(\xi,m)}$ (flip along the anti-diagonal). 
Let
$
A = T^{(\alpha,m)} - G^{(\alpha,m)}, B = T^{(\beta,m)} - G^{(\beta,m)}
$ 
with 
Assume that $B$ is invertible. Then, the GMEVP \eqref{eq:gevp} has eigenpairs $(\lambda_j, \bfx_j)$ where
\begin{equation} \label{eq:set2}
\lambda_j = \frac{ \alpha_0 + 2 \sum_{l=1}^{m} \alpha_l \cos(l j\pi h) }{ \beta_0 + 2 \sum_{l=1}^{m} \beta_l \cos(l j\pi h)  }, \quad x_{j,k} = C \sin\left( j \pi (n-k+1/2) h \right), \quad  \  j, k =1,2,\cdots, n,
\end{equation}
where $h = \frac{1}{n+1/2}$ and  $C \ne 0$ is a constant. 
\end{corollary}
\begin{proof}
This is obvious by an exchange of indices. 
\end{proof}

\begin{remark}
We note that the non-persymmetric Hankel matrix in \eqref{eq:h1} has two non-zero parts with certain patterns at the upper-left corner and the bottom-right corner of the matrix. 
In \cite{deng2021analytical}, the exact eigenpairs are derived for the cases with persymmetric Hankel matrices. There were four cases with four patterns.
Herein, Theorem \ref{thm:set1} and Corollary \ref{thm:set2} are two non-persymmetric cases with pattern pieces from \cite{deng2021analytical}. 
Following this idea of generalisation, there are in total $4\times3 = 12$ different constructions of non-persymmetric matrices where one could derive the exact eigenpairs.
For example, consider the case 
\begin{equation} \label{eq:h2}
\begin{cases}
H^{(\alpha,m)}_{j,k} & =  -\alpha_{j+k-1}, \quad   k = 1, \cdots, m-j+1, \quad j =1, \cdots, m, \quad 1\le m\le n, \\
H^{(\alpha, m)}_{n-j+1,n-k+1} & =  \alpha_{j+k+1}, \quad  k = 1, \cdots, m-j, \quad j =1, \cdots, m-1, \quad 2\le m\le n, \\
H^{(\alpha,m)}_{j,k} & = 0, \quad  \text{otherwise},
\end{cases}
\end{equation}
which can be explicitly written as
\begin{equation*}
\begin{aligned}
H^{(\alpha,m)} & = 
\begin{bmatrix}
-\alpha_1   & -\alpha_2 & \cdots   & -\alpha_m & 0 & \cdots \\
-\alpha_2    &  \cdots  & -\alpha_m & 0 & \ddots  \\
%\alpha_4  &  \cdots  & \alpha_m & 0 & \ddots  \\
\vdots  & -\alpha_m  &0 & \ddots \\
-\alpha_m &0  &\ddots  \\
0 &\ddots  \\
\vdots  \\
& & & & &  &  & \alpha_m   \\
& & & & &  & \alpha_m  & \vdots \\
& & & & &  \alpha_m & \ddots & \alpha_2 \\
& & & & \alpha_m & \cdots & \alpha_2 & \alpha_1
\end{bmatrix}_{n\times n}.
\end{aligned}
\end{equation*}
When $m=3$ and $n\ge4$, then the matrix $A=T^{(\xi,m)} + H ^{(\xi,m)}$ is of the explicit form
\begin{equation*}
\begin{aligned}
A & = 
\begin{bmatrix}
\alpha_0 - \alpha_1  & \alpha_1 - \alpha_2  & \alpha_2 - \alpha_3  & \alpha_3  \\
\alpha_1 - \alpha_2   & \alpha_0- \alpha_3 & \alpha_1  & \alpha_2  & \alpha_3   \\
\alpha_2 - \alpha_3 & \alpha_1  & \alpha_0 & \alpha_1  & \alpha_2   & \alpha_3   \\
\alpha_3  & \alpha_2  & \alpha_1  & \alpha_0 & \alpha_1  & \alpha_2   & \alpha_3   \\
&\ddots & \ddots & \ddots & \ddots & \ddots & \ddots & \ddots  \\
&  & \alpha_3  & \alpha_2  &\alpha_1  & \alpha_0 & \alpha_1 & \alpha_2+ \alpha_3 \\
&& &  \alpha_3  & \alpha_2  & \alpha_1  & \alpha_0 & \alpha_1 + \alpha_2  \\
& && &  \alpha_3  & \alpha_2 + \alpha_3 & \alpha_1 + \alpha_2  & \alpha_0 + \alpha_1   \\
\end{bmatrix}_{n\times n}.
\end{aligned}
\end{equation*}
We remark that matrices with such a pattern arise from isogeometric discretisations of the Laplacian eigenvalue problem with mixed Dirichlet-Neumann boundary conditions. 
For such a case, we have the following results. 
Let $n\ge2, 1\le m \le n-1$ and
$
A = T^{(\alpha,m)} + H^{(\alpha,m)}, B = T^{(\beta,m)} + H^{(\beta,m)}
$ 
with $T^{(\xi,m)}$ and $H^{(\xi,m)}$, $\xi = \alpha, \beta$,  defined in \eqref{eq:T} and \eqref{eq:h2}, respectively.
Assume that $B$ is invertible. Then, the GMEVP \eqref{eq:gevp} has eigenpairs $(\lambda_j, \bfx_j)$ where
\begin{equation} \label{eq:set22}
\lambda_j = \frac{ \alpha_0 + 2 \sum_{l=1}^{m} \alpha_l \cos(l j\pi h) }{ \beta_0 + 2 \sum_{l=1}^{m} \beta_l \cos(l j\pi h)  }, \quad x_{j,k} = C \sin\left( (j-1/2) \pi (k-1/2) h \right), \quad \  j, k =1,2,\cdots, n,
\end{equation}
where $h=1/n$ and $C \ne 0$ is a constant. 
The eigenvalues remain the same. This difference is on the eigenvector patterns $\sin(jk\pi h)$ or $\cos(jk\pi h)$, varying on three parameters $j, k, h$. 
One may follow this idea to develop the exact eigenpairs for the other non-persymmetric matrices. We omit the derivations here. 
\end{remark}

\subsection{A tridiagonal matrix}

We now consider a tridiagonal matrix of the pattern
\begin{equation} \label{eq:g1}
\begin{aligned}
G^{(\alpha)} & = 
\begin{bmatrix}
\alpha_0  & \alpha_1      \\
\alpha_1    & \alpha_2 & \alpha_1      \\
&  \alpha_1  & \alpha_0 & \alpha_1     \\
& &\ddots & \ddots & \ddots &  \\
&  &    &\alpha_1  & \alpha_2 & \alpha_1  \\
&& &    & \alpha_1  & \alpha_0 & \alpha_1   \\
& && &    & \alpha_1  & \alpha_2   \\
\end{bmatrix}_{n\times n},
\end{aligned}
\end{equation}
where $n$ is (apparently, due to the pattern) an even number. 
Apparently, $G^{(\alpha)}$ is symmetric and non-persymmetric.
When $n$ is odd, then the last entry is $\alpha_0$. This is a special case of the corner-overlapped block-diagonal matrices in \cite[Sect. 2.3]{deng2021analytical} and the exact eigenpairs are given in  \cite[Thm. 2.6]{deng2021analytical}.
For the case when $n$ is even, we have the following result.

\begin{theorem}[Exact eigenvalues and eigenvectors, set 3]\label{thm:set3}
Let $n\ge2$ and
$
A = G^{(\alpha)}, B = G^{(\beta)}
$ 
Assume that $B$ is invertible. Then, the GMEVP \eqref{eq:gevp} has eigenpairs $(\lambda_j, \bfx_j)$ where
\begin{equation}\label{eq:set3}
\begin{aligned}
\lambda_{2j-1}, \lambda_{2j} & = \frac{(\alpha_0 \beta_2 + \alpha_2\beta_0) \pm \sqrt{(\alpha_0 \beta_2 + \alpha_2\beta_0)^2 - 4\beta_0\beta_2 s}}{2\beta_0 \beta_2}, \\
s & = \alpha_0 \alpha_2 - 2 (\alpha_1 - \beta_1)^2 - 2(\alpha_1 - \beta_1)^2 \cos(j\pi h), \\
 \quad x_{2j-1,2k-1} & = C \sin( j \pi (k-1/2) h),  \\
  \quad x_{2j-1,2k} & = \frac{\alpha_1 - \beta_1}{\beta_2\lambda_{2j-1}- \alpha_2} (x_{2j-1, 2k-1} + x_{2j-1, 2k+1}),  \\
   \quad x_{2j,2k-1} & = C \sin( j \pi (k-1/2) h),  \\
  \quad x_{2j,2k} & = \frac{\alpha_1 - \beta_1}{\beta_2\lambda_{2j}- \alpha_2} (x_{2j, 2k-1} + x_{2j, 2k+1}),  \\
\end{aligned}
\end{equation}
where $j, k =1,2,\cdots, m, m = n/2, h = \frac{1}{m+1/2}, C\ne 0$ is some normalisation constant and $x_{j, 2m+1} = x_{j, n+1} = 0, j=1,2, \cdots,n.$
\end{theorem}

\begin{proof}
We first perform the elimination of even rows (elimination of odd rows follows similarly).  
The even rows of the GMEVP are rewritten 
\begin{equation} \label{eq:x2k}
x_{2k} = \frac{\alpha_1 - \beta_1}{\beta_2 \lambda- \alpha_2} (x_{2k-1} + x_{2k+1}), \qquad k=1,2,\cdots,m,
\end{equation}
where $x_{2m+1}=0$.
Substituting them into the odd equations to arrive at 
\begin{equation}
x_{2k-1} = \frac{\alpha_1 - \beta_1}{\beta_0 \lambda- \alpha_0} (x_{2k-2} + x_{2k}) 
= 
\begin{cases}
 \frac{(\alpha_1 - \beta_1)^2}{(\beta_0 \lambda- \alpha_0)(\beta_2 \lambda- \alpha_2) } ( x_{2k-1} + x_{2k+1}), k = 1, \\
 \frac{(\alpha_1 - \beta_1)^2}{(\beta_0 \lambda- \alpha_0)(\beta_2 \lambda- \alpha_2) } (x_{2k-3} + 2x_{2k-1} + x_{2k+1}), k=2,\cdots,m,
\end{cases}
\end{equation}
where $x_0=0$.
Let $\bfy = (x_1, x_3, \cdots, x_{2m-1})^T$. Then, the above system for odd-equations can be rewritten as the following quadratic eigenvalue problem
\begin{equation}
Q \bfy = \bfz,
\end{equation}
where
\begin{equation}
\begin{aligned}
Q = & \lambda^2 \beta_0\beta_2 I - \lambda (\alpha_0 \beta_2 + \alpha_2\beta_0) I \\
& + 
\begin{bmatrix}
ac-b^2   & -b^2    \\
-b^2   & ac-2b^2 & -b^2  \\
  & -b^2   & ac-2b^2 & -b^2   \\
& &\ddots & \ddots & \ddots \\
&& &  -b^2   & ac-2b^2    \\
\end{bmatrix}_{m\times m}
\end{aligned}
\end{equation}
with $a = \alpha_0, c = \alpha_2, b = \alpha_1 - \beta_1.$ Herein, $I$ is the identity matrix.
For this quadratic eigenvalue problem, we apply Theorem 2.7 in \cite{deng2021analytical} and Theorem \ref{thm:set1}  to get the exact eigenpairs.
In particular, the eigenvalues are the roots of
\begin{equation}
\beta_0 \beta_2 \lambda_j^2 -  (\alpha_0 \beta_2 + \alpha_2\beta_0) \lambda_j + \big( \alpha_0 \alpha_2 - 2 (\alpha_1 - \beta_1)^2 - 2(\alpha_1 - \beta_1)^2 \cos(j\pi h) \big) = 0,
\end{equation}
where $h= 1/(m+1/2) = 2/(n+1), j=1,\cdots,m$.
The eigenvectors are $\bfy_j$ with entries
\begin{equation}
y_{j,k} =x_{j,2k-1} =  C \sin\left( j \pi (k-1/2) h \right), \quad j, k= 1,\cdots, m
\end{equation}
where $C\ne 0$ is a constant. Using \eqref{eq:x2k}, one can obtain the even-th components of the eigenvectors $\bfx_j$.
This completes the proof. 
\end{proof}

%
%For $A$ defined in \eqref{eq:a1},  the MEVP $A x= \lambda x$ has the following eigenpairs
%
%
%\begin{proof}
%We first perform the elimination of even rows (elimination of odd rows follows similarly).  
%The even rows of the MEVP $A x= \lambda x$ are rewritten 
%\begin{equation}
%\begin{aligned}
%%
%x_{2k} & = \frac{b}{\lambda- c} (x_{2k-1} + x_{2k+1}), \qquad k=1,2,\cdots,m-1 \\
%x_{2m} & = \frac{b}{\lambda- c} x_{2m-1},  \\
%\end{aligned}
%\end{equation}
%which are substituted into the odd equations to obtain 
%\begin{equation}
%\begin{aligned}
%%
%a x_1 + \frac{b^2}{\lambda- c} (x_1 + x_3) & = \lambda x_1, \\
%\frac{b^2}{\lambda- c} (x_{2k-1} + x_{2k+1}) + a x_{2k+1} + \frac{b^2}{\lambda- c} (x_{2k+1} + x_{2k+3}) & = \lambda x_{2k+1}, \quad k=1,2,\cdots, m-1, 
%\end{aligned}
%\end{equation}
%where we set $x_{2m+1}=0$.
%%
%This is further rewritten to the following quadratic eigenvalue problem $Q\hat{x} = 0$ with $\hat{x} = (x_1, x_3, \cdots, x_{2m-1})^T$ and
%\begin{equation}
%\begin{aligned}
%%
%Q = \lambda^2 I - \lambda (a+c) I + 
%\begin{bmatrix}
%ac-b^2   & -b^2    \\
%-b^2   & ac-2b^2 & -b^2  \\
%  & -b^2   & ac-2b^2 & -b^2   \\
%& &\ddots & \ddots & \ddots \\
%&& &  -b^2   & ac-2b^2    \\
%\end{bmatrix}_{m\times m},
%\end{aligned}
%\end{equation}
%%
%which leads to the desired eigenpairs as a special case of combining Theorems 2.1, 2.2, and 2.7 in \cite{deng2021analytical}.
%%
%\end{proof}
%
%
%\begin{remark}
%In the case of $\alpha_0=\alpha_2, \beta_0 =\beta_2$, this result reduces to the result in \cite{deng2021analytical}.
%\end{remark}

\section{Non-symmetric and non-persymmetric matrices} \label{sec:nsnp}

In this section, we derive exact eigenpairs for less structured matrices, namely when they are both non-symmetric and non-persymmetric matrices.

\subsection{A non-symmetric tridiagonal matrix}
To start with, we note that the derivations in the proof of Theorem \ref{thm:set3} can be easily generalised to the non-symmetric case. 
We consider a tridiagonal matrix of the pattern
\begin{equation} \label{eq:g2}
\begin{aligned}
G^{(\alpha)} & = 
\begin{bmatrix}
\alpha_0  & \alpha_1      \\
\alpha_3    & \alpha_2 & \alpha_3      \\
&  \alpha_1  & \alpha_0 & \alpha_1     \\
& &\ddots & \ddots & \ddots &  \\
&  &    &\alpha_3  & \alpha_2 & \alpha_3  \\
&& &    & \alpha_1  & \alpha_0 & \alpha_1   \\
& && &    & \ddots  & \ddots & \ddots   \\
\end{bmatrix}_{n\times n}.
\end{aligned}
\end{equation}
where $n$ can be even or odd. $G^{(\alpha)}$ is non-symmetric and non-persymmetric. 
Below, we give the exact eigenpairs for the case when $n$ is an even number. 
The proof is straightforward following the proof of Theorem \ref{thm:set3}. 
We also comment that when $n$ is an odd number, the eigenpairs can be derived following the same idea.

\begin{corollary}[Exact eigenvalues and eigenvectors, set 4]\label{thm:set4}
Let $n\ge2$ be an even number and
$
A = G^{(\alpha)}, B = G^{(\beta)}
$ 
Assume that $B$ is invertible. Then, the GMEVP \eqref{eq:gevp} has eigenpairs $(\lambda_j, \bfx_j)$ where
\begin{equation}\label{eq:set4}
\begin{aligned}
\lambda_{2j-1}, \lambda_{2j} & = \frac{(\alpha_0 \beta_2 + \alpha_2\beta_0) \pm \sqrt{(\alpha_0 \beta_2 + \alpha_2\beta_0)^2 - 4\beta_0\beta_2 s}}{2\beta_0 \beta_2}, \\
s & = \alpha_0 \alpha_2 - 2 (\alpha_1 - \beta_1)^2 - 2(\alpha_1 - \beta_1)(\alpha_3-\beta_3) \cos(j\pi h), \\
 \quad x_{2j-1,2k-1} & = C \sin( j \pi (k-1/2) h),  \\
  \quad x_{2j-1,2k} & = \frac{\alpha_3 - \beta_3}{\beta_2\lambda_{2j-1}- \alpha_2} (x_{2j-1, 2k-1} + x_{2j-1, 2k+1}),  \\
   \quad x_{2j,2k-1} & = C \sin( j \pi (k-1/2) h),  \\
  \quad x_{2j,2k} & = \frac{\alpha_3 - \beta_3}{\beta_2\lambda_{2j}- \alpha_2} (x_{2j, 2k-1} + x_{2j, 2k+1}),  \\
\end{aligned}
\end{equation}
where $j, k =1,2,\cdots, m, m = n/2, h = \frac{2}{n+1}, C\ne 0$ is some normalisation constant and $x_{j, 2m+1} = x_{j, n+1} = 0, j=1,2, \cdots,n.$
\end{corollary}

%\begin{definition}[Row symmetry]
%We call the matrix is 
%\end{definition}
%

\subsection{A pentadiagonal matrix} \label{sec:b5}

We consider the following type of pentadiagonal matrix
\begin{equation} \label{eq:g3}
G^{(\alpha)}  = 
\begin{cases}
\begin{bmatrix}
\alpha_0 -  \alpha_2   & \alpha_1  & \alpha_2    \\
\alpha_3    & \alpha_0 & \alpha_3  & \alpha_2    \\
 \alpha_2   &  \alpha_1  & \alpha_0 & \alpha_1  & \alpha_2     \\
 & \alpha_2  & \alpha_3    & \alpha_0 & \alpha_3  & \alpha_2    \\
& &\ddots & \ddots & \ddots & \ddots & \ddots \\
& & & \alpha_2  & \alpha_3    & \alpha_0 & \alpha_3    \\
& & & &  \alpha_2   &  \alpha_1  & \alpha_0 - \alpha_2     \\
\end{bmatrix}_{n\times n}, \quad \text{$n$ is odd}, \\
\begin{bmatrix}
\alpha_0 -  \alpha_2   & \alpha_1  & \alpha_2    \\
\alpha_3    & \alpha_0 & \alpha_3  & \alpha_2    \\
 \alpha_2   &  \alpha_1  & \alpha_0 & \alpha_1  & \alpha_2     \\
 & \alpha_2  & \alpha_3    & \alpha_0 & \alpha_3  & \alpha_2    \\
& &\ddots & \ddots & \ddots & \ddots & \ddots \\
& & & \alpha_2   &  \alpha_1  & \alpha_0 & \alpha_1      \\
& & & & \alpha_2  & \alpha_3    & \alpha_0 - \alpha_2    \\
\end{bmatrix}_{n\times n}, \quad \text{$n$ is even}, \\
\end{cases}
\end{equation}

\begin{theorem}[Exact eigenvalues and eigenvectors, set 5]\label{thm:set5}
Let $n\ge4$ and
$
A = G^{(\alpha)}.
$ 
The MEVP \eqref{eq:evp} has eigenpairs $(\lambda_j, \bfx_j)$ where
\begin{equation}\label{eq:set5}
\begin{aligned}
\lambda_j & =  \alpha_0 + 2 \sqrt{| \alpha_1 \alpha_3| } \cos(j\pi h)  +  2 \alpha_2 \cos(2 j\pi h), \\
 \quad x_{j,2k-1} & = C \sin( j \pi k h),  \\
  \quad x_{j,2k} & = C\sqrt{|\alpha_3/\alpha_1|} \sin( j \pi k h),  \\
\end{aligned}
\end{equation}
where $ j=1,2, \cdots, n, k =1,2,\cdots, \lceil n/2 \rceil, h = \frac{1}{n+1}, C\ne 0$ is some normalisation constant.
\end{theorem}

\begin{proof}
When $\alpha_3=\alpha_1$, the exact eigenvectors given in \cite{deng2021analytical} are of the form $C\sin{j\pi kh}$. Herein, the row pattern alternates between odd and even rows, suggesting a strategy to seek for eigenvectors in the form 
\begin{equation}
x_{j,k} = 
\begin{cases}
a \sin(j\pi kh), \qquad \text{when $k$ is odd}, \\
b \sin(j\pi kh), \qquad \text{when $k$ is even}, \\
\end{cases}
\end{equation}
where $a,b\ne 0$ are to be determined.
With this form in mind, we first consider the case when $n$ is odd.
The odd rows of the MEVP \eqref{eq:evp} can be rewritten as
\begin{equation}
\begin{aligned}
\lambda a \sin(j\pi h) & = (\alpha_0 - \alpha_2) a \sin(j\pi h) + \alpha_1 b \sin(2j\pi h) +\alpha_2 a \sin(3j\pi h),  \\ 
\lambda a \sin(j\pi k h) & = \alpha_0 a \sin(j\pi k h) + \alpha_1 b \big( \sin(j\pi (k-1) h) + \sin(j\pi (k+1) h) \big) \\ 
& + \alpha_2 a \big( \sin(j\pi (k-2) h) + \sin(j\pi (k+2) h) \big),  k = 3, 5, \cdots, n-2, \\ 
\lambda a \sin(j\pi n h) & = (\alpha_0 - \alpha_2) a \sin(j\pi n h) + \alpha_1 b \sin(j\pi (n-1) h) +\alpha_2 a \sin(j\pi (n-2) h),  \\ 
\end{aligned}
\end{equation}
which, by using trigonometric identities (namely, double-angle formula, triple-angle formula, sum and difference formulas, and periodicity identities), are reduced to
\begin{equation} \label{eq:set5eigodd}
\lambda =  \alpha_0 + 2 b \alpha_1 \cos(j\pi h) /a +  2 \alpha_2 \cos(2 j\pi h), \quad k=1,3, \cdots, n.
\end{equation}

Similarly, the even rows of the MEVP \eqref{eq:evp} can be rewritten as
\begin{equation}
\begin{aligned}
\lambda b \sin(2j\pi h) & = \alpha_0 b \sin(2j\pi h) + \alpha_3 a \big( \sin(j\pi h) + \sin(3j\pi h) \big)+\alpha_2 b \sin(4j\pi h),  \\ 
\lambda b \sin(j\pi k h) & = \alpha_0 b \sin(j\pi k h) + \alpha_3 a \big( \sin(j\pi (k-1) h) + \sin(j\pi (k+1) h) \big) \\ 
& + \alpha_2 b \big( \sin(j\pi (k-2) h) + \sin(j\pi (k+2) h) \big),  k = 2, 4, \cdots, n-3, \\ 
\lambda b \sin(j\pi (n-1) h) & = \alpha_0 b \sin(j\pi (n-1) h) + \alpha_3 a (\sin(j\pi (n-2) h) + \sin(j\pi n h) ) \\
& + \alpha_2 a \sin(j\pi (n-3) h),  
\end{aligned}
\end{equation}
which reduces to
\begin{equation}\label{eq:set5eigeven}
\lambda =  \alpha_0 + 2 a \alpha_3 \cos(j\pi h) /b +  2 \alpha_2 \cos(2 j\pi h), \quad k=2,4, \cdots, n-1.
\end{equation}

We now set 
$
b \alpha_1/a = a \alpha_3 /b
$
and solve for $b$ to obtain
$b/a = \sqrt{| \alpha_3/\alpha_1 |} $ and $a/b = \sqrt{| \alpha_1/\alpha_3 |}$.
Thus, equations \ref{eq:set5eigodd} and \ref{eq:set5eigeven} reduces to
\begin{equation}\label{eq:set5eig}
\lambda =   \alpha_0 + 2 \sqrt{| \alpha_1 \alpha_3| } \cos(j\pi h)  +  2 \alpha_2 \cos(2 j\pi h),
%\alpha_0 + 2 \sum_{l=1}^{\lfloor (m+1)/2 \rfloor}  \sqrt{| \alpha_{2l-1} \hat{\alpha}_{2l-1}| } \cos((2l - 1) j\pi h)  +  2 \sum_{l=1}^{\lfloor m/2 \rfloor} \alpha_{2l} \cos(2l j\pi h) .
\end{equation}
This completes the proof for the case when $n$ is odd and the case of even number $n$ follows similarly. 
\end{proof}

\begin{remark}
Theorem \ref{thm:set5} gives the analytical solutions to MEVP rather than GMEVP. Analytical eigenpairs can be derived for GMEVP when the right-hand side matrix is symmetric and persymmetric so that the denominators in the corresponding representations \ref{eq:set5eigodd} and \ref{eq:set5eigeven} are the same. We omit this derivation here for simplicity. 
Another remark is that other boundary patterns could be considered here. For example, one can consider the pattern in Theorem \ref{thm:set1} and derive the exact eigenpairs where the eigenvectors are of the form $C \sin\left( j \pi (n-k+1/2) h \right)$ for odd entries and $C\sqrt{|\alpha_3/\alpha_1|} \sin( j \pi (k-1/2) h)$ for even entries. 
\end{remark}

\subsection{Generalisation of the pentadiagonal matrix}

We consider two lines of generalisation of the result in Section \ref{sec:b5}. 
First, we consider 
\begin{equation}  \label{eq:g4}
\begin{aligned}
G^{(\alpha,m)} & = 
\begin{bmatrix}
\alpha_0 - \alpha_2   & \alpha_1-\alpha_3   & \cdots   & \alpha_{m-2} - \alpha_m & \alpha_{m-1} & \alpha_m \\
%\hat{\alpha}_1 -  \hat{\alpha}_3   & \alpha_0 - \alpha_4 & \hat{\alpha}_1 -  \hat{\alpha}_5  & \cdots  & \alpha_m   \\
\vdots  & \vdots  & \vdots & \vdots  & \vdots   & \vdots  & \ddots \\
%\vdots  & \alpha_1  & \alpha_0 & \alpha_1  & \cdots   & \alpha_m   \\
%
%
  \cdots  & \alpha_1  & \alpha_0 & \alpha_1  & \alpha_2 & \alpha_3 & \cdots   & \alpha_m  &  \\
&   \cdots &  \hat{\alpha}_1  & \alpha_0 & \hat{\alpha}_1  & \alpha_2 & \hat{\alpha}_3 & \cdots   & \alpha_m  & \\
& &\ddots & \ddots & \ddots & \ddots & \ddots & \ddots & \ddots  & \vdots  \\
\end{bmatrix}_{n\times n},
\end{aligned}
\end{equation}
where internal rows are alternating similarly as for the pentadiagonal matrix in Section \ref{sec:b5}, and the boundary entries are set accordingly so that the reduced eigenvalue forms are consistent in each row. 
We have the following result and the proof is omitted as it is trivial following the proof of Theorem \ref{thm:set5}.

\begin{corollary}[Exact eigenvalues and eigenvectors, set 6]\label{thm:set6}
Let $n\ge4, n-2\ge m\ge2$ and
$
A = G^{(\alpha, m)}.
$ 
Assuming $\frac{\hat{\alpha}_1}{\alpha_1} = \frac{\hat{\alpha}_3}{\alpha_3} = \cdots = \frac{\hat{\alpha}_{2 \lfloor (m-1)/2 \rfloor +1} }{\alpha_{2 \lfloor (m-1)/2 \rfloor +1}}$.
The MEVP \eqref{eq:evp} has eigenpairs $(\lambda_j, \bfx_j)$ where
\begin{equation}\label{eq:set6}
\begin{aligned}
\lambda_j & = \lambda =  \alpha_0 + 2 \sum_{l=1}^{\lfloor (m+1)/2 \rfloor}  \sqrt{| \alpha_{2l-1} \hat{\alpha}_{2l-1}| } \cos((2l - 1) j\pi h)  +  2 \sum_{l=1}^{\lfloor m/2 \rfloor} \alpha_{2l} \cos(2l j\pi h),  \\
 \quad x_{j,2k-1} & = C \sin( j \pi k h),  \\
  \quad x_{j,2k} & = C\sqrt{|\hat{\alpha}_1/\alpha_1|} \sin( j \pi k h),  \\
\end{aligned}
\end{equation}
where $ j=1,2, \cdots, n, k =1,2,\cdots, \lceil n/2 \rceil, h = \frac{1}{n+1}, C\ne 0$ is some normalisation constant.
\end{corollary}

\begin{remark}
Another way of generalisation is to consider eigenvectors of a more general form, for example, in the form
\begin{equation}
x_{j,k} = 
\begin{cases}
a \sin(j\pi kh), \qquad \text{if $\mod(k,3) =0$}, \\
b \sin(j\pi kh), \qquad \text{if $\mod(k,3) =1$}, \\
c \sin(j\pi kh), \qquad \text{if $\mod(k,3) =2$}, \\
\end{cases}
\end{equation}
for three-row patterns, or
$$
a_j \sin(j\pi kh), \qquad \text{if $\mod(k,p) =0$}, 
$$ 
for $p$-row matrix patterns.
These generalisations correspond to matrices of certain structures, addressing an inverse-type problem within the field of matrix eigenvalue analysis. In this context, the task involves generating matrices based on predefined forms of eigenvectors. Specifically, one may start with an assumption about the desired form of the eigenvectors and then work to determine the matrices that would naturally exhibit these vectors as their eigenvectors.

Given the vast potential configurations, there are likely infinitely many matrices that could satisfy the conditions for the assumed eigenvectors. However, to achieve uniqueness in the solutions and refine the scope of our inquiry, it may be necessary to impose additional structural constraints on the matrices. Currently, the consideration of such constraints to ensure unique matrix determination remains an open area for further research. This aspect of our study lays the groundwork for future explorations aimed at defining more precise conditions under which these matrices can be uniquely constructed from specified eigenvector forms.
\end{remark}

\section{Concluding remarks} \label{sec:conclusion} 
This paper has expanded the domain of exact analytical solutions for matrix eigenvalue problems by addressing the complexities of non-symmetric and non-persymmetric matrices. Such matrices are prevalent in one-dimensional Laplacian eigenvalue problems with mixed boundary conditions and several applications in quantum mechanics where traditional symmetric or persymmetric matrix assumptions do not suffice. Through rigorous mathematical development, we have derived exact eigenpairs for multiple sets of structured matrices, significantly broadening the applicability and relevance of these analytical methods in physics and engineering simulations.

Our results not only provide a robust foundation for future research in the field but also enhance computational methodologies by offering exact solutions to matrix eigenvalue problems. The implications of this work are substantial, offering new perspectives and tools for scientists and engineers working with complex matrix configurations in various applications. As matrix problems continue to evolve with new scientific challenges, the methodologies developed herein will serve as a resource for addressing these advanced problems.

\section*{Acknowledgments} 
The author thanks Dr. Anas Abdelwahab for questions and discussions on several special matrices arisen in Quantum Physics, as well as for the suggestion to document these new results in a manuscript. 
%Dr. Abdelwahab's insights were crucial in adapting our analytical methods to encompass a wider range of matrices relevant to quantum physics applications, thereby extending the research's relevance to a broader scientific community.
%Dr. Abdelwahab's perspectives were instrumental in adapting our analytical methods to a broader set of matrices relevant to quantum physics applications, making the research applicable to a wider scientific audience. 

%The author thanks Anas Abdelwahab for the insightful and motivating discussions on generalising the results in \cite{deng2021analytical} to various matrices arising in Quantum Physics. With his suggestion for a more broad readers, this manuscript is prepared. 

%\section*{References}

%\bibliographystyle{plain}
\bibliographystyle{siam}
%\bibliography{ref.bib}
%\bibliographystyle{siamplain}
%\bibliographystyle{siam}
%\bibliography{ref}

%We present other 

%\appendix{\textbf{Matrices}}
%

%\bibliographystyle{siamplain}
\bibliography{ref}
\end{document}